\documentclass[12pt]{amsart}
\usepackage[english]{babel}
\usepackage[utf8]{inputenc}
\usepackage{lipsum}
\usepackage{mathrsfs}
\usepackage{stix}
\usepackage{fullpage}
\usepackage{mathtools}
\usepackage{amsmath}
\usepackage{tikz-cd}
\usepackage[colorlinks]{hyperref}
\usepackage{fullpage}
\usepackage{mathtools}
\usepackage{amsmath}
\usepackage{tikz-cd}
\numberwithin{equation}{section}
\theoremstyle{plain}
\newtheorem{theorem}{Theorem}[section]

\newtheorem{rem}[theorem]{Remark}

\def\X{\mathbb{Y}}
\allowdisplaybreaks
\begin{document}
\title[A note on best constants for sharp Weighted Integral Hardy inequalities ]{A note on best constants for Weighted Integral Hardy inequalities on homogeneous groups}

\author[Michael Ruzhansky]{Michael Ruzhansky}
\address{
	Michael Ruzhansky:
	\endgraf
	Department of Mathematics: Analysis, Logic and Discrete Mathematics
	\endgraf
	Ghent University, Belgium
	\endgraf
	and
	\endgraf
	School of Mathematical Sciences
	\endgraf
	Queen Mary University of London
	\endgraf
	United Kingdom
	\endgraf
	{\it E-mail-} {\rm michael.ruzhansky@ugent.be}
}

\author[A. Shriwastawa]{Anjali Shriwastawa}
\address{
 Anjali Shriwastawa:
 \endgraf
  DST-Centre for Interdisciplinary Mathematical Sciences  \endgraf
  Banaras Hindu University, Varanasi-221005, India
  \endgraf
  {\it E-mail-} {\rm anjalisrivastava7077@gmail.com}}

\author[B.Tiwari]{Bankteshwar Tiwari}
\address{
 Bankteshwar Tiwari :
 \endgraf
  DST-Centre for Interdisciplinary Mathematical Sciences  \endgraf
  Banaras Hindu University, Varanasi-221005, India
  \endgraf
  {\it E-mail-} {\rm banktesht@gmail.com}}


\subjclass[2010]{26D10, 22E30, 45J05.}

\begin{abstract} The main aim of this note is to prove sharp weighted integral Hardy inequality and conjugate integral Hardy inequality on homogeneous Lie groups  with any quasi-norm for the range $1<p\leq q<\infty.$ We also calculate the precise value of sharp constants in respective inequalities, improving the result of \cite{RV} in the case of homogeneous groups.
\end{abstract}
\maketitle
\allowdisplaybreaks
\section{Introduction}

In his seminal papers \cite{HD1,HD2}, G. H. Hardy stated and proved in 1920  and   1925, respectively, the following  inequality, which is now known as ``(integral) Hardy inequality": For a non-negative and measurable function $f$ on $(0,\infty)$, we have
\begin{align}\label{HD11}
    \int_0^\infty\left(\frac{1}{x}\int_0^x f(t)\, dt\right)^pdx\le \left(\frac{p}{p-1}\right)^p\int_0^\infty f^p(x)\, dx,\quad p>1.
\end{align}
The first generalization of \eqref{HD11} was presented by Hardy \cite{HD3} in 1927, which is a weighted form of \eqref{HD11} and state  that: If $f$ is non-negative and measurable function on $(0,\infty)$, then
\begin{align}\label{HD12}
 \int_0^\infty\left(\frac{1}{x}\int_0^x f(t)dt\right)^p x^\alpha \,dx\le \left(\frac{p}{p-1-\alpha}\right)^p\int_0^\infty f^p(x) x^\alpha \,dx,\quad p\ge 1,\quad \alpha<p-1.   
\end{align}

The further developments and improvements of \eqref{HD11} and \eqref{HD12}  are discussed in several books, monographs and papers; we refer  \cite{KM1,KP, KM2, KM3,PS,OK,D} and references therein for more details. Let us first mention the following modern form of weighted version of Hardy's original inequality:
\begin{align}\label{HD13}
  \left(\int_0^\infty\left(\int_0^x f(t)dt\right)^q u(x)\, dx\right)^\frac{1}{q}\le C \left(\int_0^\infty f^p(x) v(x)\, dx\right)^\frac{1}{p},   
\end{align}
where $f(x)\ge 0$, $u$ and $v$ are weights functions  and $1\le p,q <\infty$. In this paper, we consider the case $1 < p\le q < \infty$. 
\newline The constants in \eqref{HD11}  and \eqref{HD12} are sharp. To find a best constant $C$ in \eqref{HD13} is an intriguing and challenging problem. In the case  of power weights, the analysis of the best constant of \eqref{HD13} is given in \cite{KP,PS}. 
Since then a lot of work has been done on Hardy inequalities in different forms and in different settings on higher dimensional Euclidean space. It is clearly impossible to give a complete overview of the literature, so let us only refer to books and surveys \cite{OK, D,KPS,EE} and references therein. 
 The sharp constants in Hardy type inequalities on the Euclidean space are known only in a few cases. More precisely, Perrson and Samko \cite{PS} derived  sharp weighted integral Hardy inequalities on  Euclidean space with power weights.
 
 In this article, our main objects are homogeneous (Lie) groups of homogeneous dimension $Q$ equipped with a quasi-norm
$|\cdot|$. By definition, a homogeneous Lie group is a Lie group equipped with a family of
dilations compatible with the group law. For the general description of the set-up of homogeneous groups, we refer to \cite{FS,FR, RS}.
Particular examples of homogeneous groups are the Euclidean space $\mathbb{R}^n$ (in which case $Q = n$), the
Heisenberg group, as well as general stratified groups (homogeneous Carnot groups) and graded groups. Recently, Hardy type inequalities and their best constants have been extensively investigated in non-commutative settings (e.g. Heisenberg groups, graded groups, homogeneous groups); we cite \cite{FL12,RS, RY, RS1} just to a mention a few of them.

More generally,  Ruzhansky and Verma \cite{RV} obtained several characterizations of weights for two-weight Hardy inequalities to hold on general metric measure spaces possessing polar decompositions for the range $1<p \leq q<\infty$ (see, \cite{RV1} for the case $0<q<p$ and $1<p<\infty$). As a consequence, one deduced new weighted Hardy inequalities on $\mathbb{R}^n$, on homogeneous groups, on hyperbolic spaces and on Cartan–Hadamard manifolds \cite{RV}. In particular, one proved the following integral Hardy inequality \cite[Corollary 3.1]{RV} on homogeneous groups.
\begin{theorem} \label{RVtheo}
Let $\mathbb{G}$ be a homogeneous group of homogeneous dimension $Q$, equipped with a quasi norm $|\cdot|$. Let $1<p\le q<\infty$ and  let $\alpha, \beta \in \mathbb{R},$ then the inequality 
\begin{align}\label{EQ1}
    \left(\int_\mathbb{G}\left(\int_{\mathbb{B}(0, |x|)}|u(y)| dy \right)^q|x|^\beta dx\right)^\frac{1}{q}\le C\left(\int_\mathbb{G}|u(x)|^q |x|^\alpha dx \right)^\frac{1}{p}
\end{align}
holds for all measurable functions $u$ on $\mathbb{G}$ if and only if \begin{align}\label{thm1.1}
\beta+Q<0,\quad \alpha<Q(p-1)\quad\text{and}\quad 
q(\alpha+Q)-p(\beta+Q)=pqQ.
\end{align}
 Moreover, the constant $C$ for  \eqref{EQ1} satisfies
 \begin{align} \label{shrco}
   \frac{\mathfrak{S}^{\frac{1}{q}+\frac{1}{p'}}}{|\beta +Q|^\frac{1}{q}\left( \alpha\left(1-p'\right)+Q\right)^\frac{1}{p'}} \le C\le (p')^\frac{1}{p'} (p)^\frac{1}{q} \frac{\mathfrak{S}^{\frac{1}{q}+\frac{1}{p'}}}{|\beta +Q|^\frac{1}{q}\left( \alpha\left(1-p'\right)+Q\right)^\frac{1}{p'}}  \end{align}
   where $\mathfrak{S}$ is the area of the unit sphere in $\mathbb{G}$ with respect to the quasi-norm $|\cdot|$.
\end{theorem}

In the above Theorem \ref{RVtheo} , the constant appearing in inequality \eqref{EQ1} is bounded by upper and lower bound of certain quantities as given in \eqref{shrco} but may not be sharp, in general. The main aim of this paper is to fill this gap and to obtain a sharp version of \eqref{EQ1} with the precise value of the sharp constant. We also prove a sharp weighted conjugate Hardy inequality on Homogeneous (Lie) groups. For the proof, we follow the method developed in \cite{PS} in the (isotropic and abelian) setting of  Euclidean space.  We note that also in the abelian (both isotropic and anisotropic) cases of $\mathbb{R}^n$, our results provide new insights in view of the arbitrariness of the quasi-norm $|\cdot|$ which does not necessarily have to be the Euclidean norm.

In this next section, we present basic definition, notation and terminologies related with homogeneous Lie groups and some basic result concerning to the one dimensional sharp Hardy inequalities. In the last section, we discuss our main results on sharp weighted Hardy inequalities on homogeneous groups and obtain the precise value of the sharp constant.

Throughout this paper, the symbol $A\asymp B$ means $\exists\,C_{1},C_{2}>0$ such that $C_{1}A\leq B\leq C_{2}A$.

\section{Preliminaries} \label{preli}

In this section, we recall the basics of homogeneous groups and Hardy inequalities on Euclidean space. For more detail on homogeneous groups as well as several functional inequalities on homogeneous groups, we refer to  monographs \cite{FR, FS,RS} and references therein.  For Hardy inequalities on Euclidean space, one can see \cite{PS,KPS,KP} and references therein.

\subsection{Basics on homogeneous groups} 

A Lie group $\mathbb{G}$ (identified with $(\mathbb{R}^N, \circ)$) is called a homogeneous group if it is equipped  with a dilation mapping $$ D_\lambda:\mathbb{R}^N \rightarrow \mathbb{R}^N,\quad\lambda >0,$$ defined as 
\begin{equation}
 D_\lambda(x)= (\lambda^{v_1}x_1,\lambda^{v_2}x_2,\ldots,\lambda^{v_N}x_N) ,\quad v_1,v_2,\dots,v_N > 0,
\end{equation} 
which is an automorphism of the group $\mathbb{G}$ for each $\lambda >0 $. At times, we will denote the image of $x\in \mathbb{G}$ under $D_\lambda$ by $\lambda(x)$ or, simply $\lambda x$.
The homogeneous dimension $Q$ of a homogeneous group $\mathbb{G}$ is defined by 
$$Q=v_1+v_2+\dots+v_N.$$
It is well known that a homogeneous group is necessarily nilpotent and unimodular. The Haar measure $dx$ on $\mathbb{G}$ is nothing but the Lebesgue measure on $\mathbb{R}^N$. 

Let us denote the Haar volume of a measurable set $\omega \subset \mathbb{G}$ by $|\omega|$. Then we have the following consequences: for $\lambda >0$
\begin{equation}
    |D_\lambda(\omega)|=\lambda^Q |\omega|  \quad \text{and} \quad \int_{\mathbb{G}} f(\lambda x) dx = \lambda^{-Q}\int_{\mathbb{G}} f(x) dx.
\end{equation}
A quasi-norm on $\mathbb{G}$ is any continuous non-negative function $ |\cdot|:\mathbb{G} \rightarrow [0,\infty)$ satisfying the following conditions:
\begin{itemize}
    \item[(i)] $|x|=|x^{-1}|$ for all $x \in \mathbb{G},$
    \item[(ii)] $|\lambda x|=\lambda |x|$\, for all \, $x \in \mathbb{G}$ and $\lambda >0,$
    \item[(iii)] $|x|=0 \iff x=0.$
\end{itemize}

If $\mathfrak{S}= \{x \in \mathbb{G}: |x|=1\} \subset \mathbb{G}$ is the unit sphere with respect to the quasi-norm, then there is a unique Radon measure $\sigma$ on $\mathfrak{S}$ such that for all $f \in L^1(\mathbb{G})$, we have the followoing polar decomposition
\begin{equation} \label{polar}
    \int_{\mathbb{G}} f(x) dx =\int_0^\infty \int_\mathfrak{S} f(ry) r^{Q-1}d\sigma(y) dr.
\end{equation}

Firstly, let us consider the metric space $(\mathbb X,d)$ with a Borel measure $dx$ allowing for the following {\em polar decomposition} at $a\in{\mathbb X}$: we assume that there is a locally integrable function $\lambda \in L^1_{loc}$  such that for all $f\in L^1(\mathbb X)$ we have
   \begin{equation}\label{EQ:polarintro}
   \int_{\mathbb X}f(x)dx= \int_0^{\infty}\int_{\Sigma} f(r,\omega) \lambda(r,\omega) d\omega_{r} dr,
   \end{equation}
    for some set $\Sigma=\{x\in\mathbb{X}:d(x,a)=r\}\subset \mathbb X$ with a measure on it denoted by $d\omega$, and $(r,\omega)\rightarrow a $ as $r\rightarrow0$.
    
 Similar to the Euclidean spaces, we define the weighted Hardy operator and the weighted conjugate Hardy operator  on a homogeneous group $\mathbb{G}$ by
    \begin{equation}{\label{WH}}
       Hu(x)=|x|^{\alpha-Q} w(|x|)\int_{\mathbb{B}(0,|x|)}
       \frac{u(y)}{w(|y|)} dy,
       \end{equation}
    \begin{equation}
    Hu(x)=|x|^\alpha w(|x|)\int_{\mathbb{G} \backslash \mathbb{B}(0,|x|)}
       \frac{u(y)}{|y|^Q w(|y|)} dy,
    \end{equation}
where $\alpha \ge 0$ and $w$ is a radial weight on $\mathbb{G}.$ In particular, for $\mu \in \mathbb{R}$ the function $w(x):=|x|^\mu, \forall x \in \mathbb{G}$, defined a weight function on $\mathbb{G}.$

Now, we recall the one dimensional sharp weighted Hardy and conjugate Hardy inequalities established in \cite[Theorem 2.7]{PS}. This result will be used to prove our main result.
\begin{theorem}\label{Dim1}
Let $1<p\le q<\infty$. Then the following statements $(a)$ and $(b)$ hold and are equivalent:

\begin{itemize}
\item[(a)] The inequality 
\begin{equation}\label{EQ:Hardy3}
     \left(\int_0^\infty \left(\int_0^x u(t) dt \right)^q x^\beta dx\right)^\frac{1}{q}\le D_{p,q, \alpha}\left(\int_0^xu^p(x) x^\alpha dx\right)^\frac{1}{p}
 \end{equation}
holds for all measurable functions $u$ on $(0,\infty)$ if and only if
\begin{equation}\label{EQ:H3}
\alpha < p-1 \quad and \quad q(\alpha+1)-p(\beta+1)=pq
\end{equation}
\item[(b)] The inequality 
\begin{equation}\label{EQ:Hardy4}
\left(\int_0^\infty \left(\int_x^\infty u(t) dt \right)^q x^{\beta_0} dx\right)^\frac{1}{q}\le D_{p,q,\alpha}\left(\int_0^\infty u^p(x) x^{\alpha_0} dx\right)^\frac{1}{p}
\end{equation} 
holds for all measurable functions $u$ on $(0,\infty)$ if and only if
\begin{equation}\label{EQ:H4}
{\alpha_0}> p-1 \quad and \quad q(\alpha_0+1)-p(\beta_0+1)=pq,
\end{equation} 
and also
\item[(c)]
The relation between the parameters $\alpha$ and $\alpha_0$ is
\begin{equation}
\alpha_0=-\alpha-2+2p
\end{equation}
and  the best constants in \eqref{EQ:Hardy3} and \eqref{EQ:Hardy4} are same.
\end{itemize}
\end{theorem}
\begin{theorem} \label{Sharpconst}
Let $1<p<q<\infty$ and let the parameters $\alpha$ and $\beta$ satisfy \eqref{EQ:H3}. Then the sharp constant $D_{p,q,\alpha}$  in \eqref{EQ:Hardy3} is given by 
\begin{equation} \label{sharp<q}
D_{p,q,\alpha}=\left ( \frac{p-1}{p-1-\alpha}\right)^{\frac{1}{p'}+\frac{1}{q}} \left( \frac{p'}{q}\right)^\frac{1}{p} \left( \frac{\frac{q-p}{p}\Gamma \left(\frac{pq}{q-p} \right)}{\Gamma \left( \frac{p}{q-p} \right)\Gamma \left( \frac{p(q-1)}{q-p}\right)}\right)^{\frac{1}{p}-\frac{1}{q}},
\end{equation}
when $p<q$ and
\begin{equation}
 D_{p,p,\alpha}=\lim_{p \rightarrow q} D_{p,q,\alpha} =\frac{p}{p-1-\alpha},  
\end{equation} when $p=q.$
\end{theorem}

\section{Main results}
This section is devoted to establishing sharp weighted Hardy and conjugate Hardy inequalities on homogeneous groups. We begin with the following sharp weighted integral Hardy inequality on homogeneous groups for the range $1<p \leq q<\infty$. 

\begin{theorem} \label{sharphardy}
Let $\mathbb{G}$ be a homogeneous group with homogeneous dimension  $Q$ equipped with  quasi norm $|\cdot|$. Let $1<p\leq q<\infty$  and let $\alpha,\beta \in \mathbb{R}$. Then, the following inequality  
\begin{align}\label{sharphardy1}
    \left( \int_{\mathbb{G}} \left| \left( \int_{\mathbb{B}(0,\, |x| )} u(y)\, dy\right)\right|^q |x|^{\beta} dx\right)^{\frac{1}{q}} \leq C(p, q, Q,\alpha) \left(\int_{\mathbb{G}} |u(x)|^p |x|^{\alpha} dx \right)^{\frac{1}{p}}
\end{align} holds  for all measurable functions $u$ on $\mathbb{G}$ if and only if 
\begin{align} \label{condn}
    \alpha<Q(p-1)\quad \text{and}\quad q(\alpha+Q)-p(\beta+Q)=pq Q.
\end{align} Moreover, the sharp constant $C(p, q, Q, \alpha)$ is given by
\begin{align}
 C(p,q,Q,\alpha)= |\mathfrak{S}|^{1+\frac{1}{q}-\frac{1}{p}} \left( \frac{p-1}{Q(p-1)-\alpha}\right)^{\frac{1}{p'}+\frac{1}{q}}\left(\frac{p'}{q}\right)^{\frac{1}{p}} \left( \frac{\frac{q-p}{p}\Gamma \left(\frac{pq}{q-p} \right)}{\Gamma \left( \frac{p}{q-p} \right)\Gamma \left( \frac{p(q-1)}{q-p}\right)}\right)^{\frac{1}{p}-\frac{1}{q}}
\end{align}
when $q>p$ and 
\begin{align} \label{Cont}
C(p, p, Q, \alpha)= \lim_{p\rightarrow q} C(p,q,Q,\alpha)=\frac{p|\mathfrak{S}|}{Q(p-1)-\alpha}, 
\end{align}
when  $p=q$, where $\mathfrak{S}= \{x \in \mathbb{G}: |x|=1\} \subset \mathbb{G}$ is the unit sphere with respect to the quasi-norm $|\cdot|,$ and $|\mathfrak{S}|$ denotes the measure of the unit sphere $\mathfrak{S}$ in the homogeneous group  $\mathbb{G}$ with respect to the quasi-norm $|\cdot|$.
\end{theorem}
\begin{rem}
The condition \eqref{condn} in  Theorem \ref{sharphardy1}  automatically implies that $\beta+Q<0$ in \eqref{thm1.1}  of Theorem \ref{RVtheo}, so the conditions on indices are consistent in  both theorems.
\end{rem}
\begin{proof}
 For the proof of the sharp Hardy inequality \eqref{sharphardy1} we use the  weighted Hardy operator on a homogeneous group $\mathbb{G}$ with quasi norm $|\cdot|$, given by  \eqref{WH}, in the following  form  
 \begin{equation}
     (H u)(x)=|x|^{\lambda+\mu-Q}\int_{\mathbb{B}(0,|x|)}\frac{u(y)}{|y|^\mu} dy, \quad \lambda \geq 0 ,
 \end{equation}
 where $\frac{\lambda}{Q}=\frac{1}{p}-\frac{1}{q}$ and $\mu \in \mathbb{R}$ will be chosen later accordingly.
  We will calculate the $L^p$-$L^q$ norm of the weighted Hardy operator $H$ in the range $1<p\leq q<\infty$ and eventually conclude the proof of \eqref{sharphardy1}.  For this purpose, let us estimate the following quantity:
  \begin{equation}\label{WH1}
   \|Hu\|_{L^q(\mathbb{G})}= \left(\int_\mathbb{G} \left||x|^{\lambda+\mu-Q}\int_{\mathbb{B}(0,|x|)}\frac{u(y)}{|y|^\mu} dy \right|^q dx\right)^\frac{1}{q}. 
  \end{equation}

Using the polar decomposition \eqref{polar} in the setting of  homogeneous groups  \eqref{WH1} will take the following form with $x=r\sigma' \, \text{and}\, y=\rho \sigma$:
\begin{align}\label{POLAR}
    &\left(\int_\mathbb{G} \left||x|^{\lambda+\mu-Q}\int_{\mathbb{B}(0,|x|)}\frac{u(y)}{|y|^\mu} dy \right|^q dx\right)^\frac{1}{q} \nonumber\\&=|\mathfrak{S}|^\frac{1}{q}\left(\int_0^\infty\left|r^{\frac{Q-1}{q}}r^{\lambda+\mu-Q} \int_o^r \int_\mathfrak{S} u(\rho \sigma)  \rho^{Q-1-\mu} d\rho d\sigma\right|^q dr \right)^\frac{1}{q}.
\end{align}
Now, by setting 
 $$ U(\rho)=\int_\mathfrak{S} u(\rho \sigma) d\sigma $$  in the above identity \eqref{POLAR}, we obtain
\begin{align}
   & \left(\int_\mathbb{G} \left||x|^{\lambda+\mu-Q}\int_{\mathbb{B}(0,|x|)}\frac{u(y)}{|y|^\mu} dy \right|^q dx\right)^\frac{1}{q} \nonumber\\&=|\mathfrak{S}|^\frac{1}{q}\left(\int_0^\infty \left| r^\frac{Q-1}{q}r^{\lambda+\mu-Q}  \int_0^r  U(\rho)\,\rho^{Q-1-\mu}d\rho\right|^q dr \right)^\frac{1}{q}
 \nonumber\\&=|\mathfrak{S}|^\frac{1}{q}\left(\int_0^\infty \left| r^\frac{Q-1}{q}r^{\lambda+\mu-Q}  \int_0^r U(\rho)\,\rho^{Q(\frac{1}{p}+\frac{1}{p'})-(\frac{1}{p}+\frac{1}{p'})-\mu}\,d\rho\right|^q dr \right)^\frac{1}{q},
 \end{align}
 where $p'$ is the Lebesgue conjugate of $p$, i.e., $\frac{1}{p}+\frac{1}{p'}=1.$
 Thus, we have
\begin{align}\label{eq3.9}
&\left(\int_\mathbb{G} \left||x|^{\lambda+\mu-Q}\int_{\mathbb{B}(0,|x|)}\frac{u(y)}{|y|^\mu} dy \right|^q dx\right)^\frac{1}{q} \nonumber
\\&=|\mathfrak{S}|^\frac{1}{q}\left(\int_0^\infty \left| r^{\frac{Q-1}{q}+\lambda+\mu-Q}  \int_0^r \,\rho^{\frac{Q-1}{p'}-\mu}\,\rho^{\frac{Q-1}{p}} \,U(\rho)\,d\rho\right|^q dr \right)^\frac{1}{q}    
\end{align}

By denoting $\delta=\frac{\lambda}{Q}$ and $\mu=\gamma+\frac{Q-1}{p'}$ and performing simple calculations 
\begin{align*}
    &\frac{Q-1}{q}+\lambda+\mu-Q\nonumber\\&=\frac{Q-1}{q}+\lambda+\gamma+\frac{Q-1}{p'}-Q\nonumber\nonumber\\&=\left(Q-1\right)\left(\frac{1}{p}-\frac{\lambda}{Q}\right)+\lambda+\gamma+\frac{Q-1}{p'}-Q
    \quad\quad\quad\left(\textnormal{since}\, \frac{\lambda}{Q}=\frac{1}{p}-\frac{1}{q}\right)\nonumber\\&=\left(Q-1\right)\left(\frac{1}{p}+\frac{1}{p'}\right)-\frac{\lambda(Q-1)}{Q}+\lambda+\gamma-Q\nonumber\\&=Q-1-\lambda+\frac{\lambda}{Q}+\lambda+\gamma-Q\quad\quad\quad\quad\quad\left(\textnormal{as}\, \frac{1}{p}+\frac{1}{p'}=1\right)\nonumber\\&=\delta+\gamma-1.
\end{align*}
So that we have
\begin{align}
  \frac{Q-1}{q}+\lambda+\mu-Q=\delta+\gamma-1,
\end{align}
and therefore, from \eqref{eq3.9} we get
\begin{align}\label{W2}
&\left(\int_\mathbb{G} \left||x|^{\lambda+\mu-Q}\int_{\mathbb{B}(0,|x|)}\frac{u(y)}{|y|^\mu} dy \right|^q dx\right)^\frac{1}{q} \nonumber \\& \quad\quad\quad=
    |\mathfrak{S}|^\frac{1}{q}\left(\int_0^\infty \left| r^{\delta+\gamma-1}  \int_0^r \rho^{-\gamma}\,\rho^{\frac{Q-1}{p}} U(\rho)\,d\rho\right|^q dr \right)^\frac{1}{q} .  
\end{align}
It follows from the relation $\mu=\gamma+\frac{Q-1}{p'}$ that \begin{equation} \label{anjliq}
    \mu<\frac{Q}{p'}\iff \gamma <\frac{1}{p'}.
\end{equation}

Now, we recall the one dimensional Hardy inequality from Theorem \ref{Dim1}: For $ 1<p\le q<\infty $, the inequality 
\begin{align}\label{OH}
    \left(\int_0^\infty \left(\int_0^x u(t) dt\right)^q x^{\beta'} \,dx\right)^\frac{1}{q} \le D_{p,q,\alpha'}\left(\int_0^x u^p(x) x^{\alpha'} \, dx\right)^\frac{1}{p}
\end{align} holds for all measurable functions $u$ on $(0,\infty)$ if and only if
\begin{equation}
    \alpha'<p-1\, \text{and}\,q(\alpha'+1)-p(\beta'+1)=pq
\end{equation} with the sharp constant $D_{p,q,\alpha'}$ (when $p<q$ as well as $p=q$) given by Theorem \ref{Sharpconst}.

 For our convenience we rewrite \eqref{OH} with  both the weight functions on the left hand side:
 
 \begin{align}\label{SS}
 \left(\int_0^\infty\left|x^{\beta'/q}\int_0^x \frac{t^{\alpha'/p}u(t)}{t^{\alpha'/p}}\, dt\right|^q \, dx\right)^\frac{1}{q}\le D_{p,q,\gamma}\left(\int_0^\infty (x^{\alpha'/p}u(x))^p\, dx\right)^\frac{1}{p}.
\end{align}
 For that we change the  notation: $\frac{\beta'}{q}=\delta+\gamma-1,\,\alpha'=\gamma p,\,\text{and}\,u(t) t^{\gamma}=g(t)$ to get
\begin{align}\label{S}
 \left(\int_0^\infty\left|x^{\delta+\gamma-1}\int_0^x\frac{g(t)}{t^{\gamma}}\, dt\right|^q \, dx\right)^\frac{1}{q}\le D_{p,q,\gamma}\left(\int_0^\infty g^p(x)\, dx\right)^\frac{1}{p};\quad \gamma<\frac{1}{p'},
\end{align} where $1<p<q<\infty$ and $\delta=\frac{1}{p}-\frac{1}{q}$ as $\delta=\frac{\lambda}{Q}$. Here the sharp constant $D_{p,q,\gamma}$ is given by $D_{p,q,\alpha'}$ as in \eqref{sharp<q} by replacing $\alpha'$ by $\gamma p.$

Next, the one dimensional sharp Hardy inequality \eqref{S} is applicable on right hand side of identity \eqref{W2} with the function $g(\rho)=\rho^\frac{Q-1}{p}U(\rho)$ and the relation \eqref{anjliq} by recalling the relation here for the inequalities in \eqref{condn}:  $$\alpha<Q(p-1)\quad\text{and}\quad q(\alpha+Q)-p(\beta+Q)=pqQ$$ holds if and only if  $$\alpha'<p-1\quad\text{and}\quad q(\alpha'+1)-p(\beta'+1)=pq$$ with $\alpha'=\alpha/Q$ and $\beta'=\beta/Q.$
Therefore, from \eqref{W2} we obtain the following sharp inequality
\begin{align}\label{HD}
 &\left(\int_\mathbb{G} \left||x|^{\lambda+\mu-Q}\int_{\mathbb{B}(0,|x|)}\frac{u(y)}{|y|^\mu} dy \right|^q dx\right)^\frac{1}{q} \le 
 D_{p,q,\gamma} |\mathfrak{S}|^\frac{1}{q}\left(\int_0^\infty \rho^{Q-1} U^p(\rho)\,d\rho\right)^{\frac{1}{p}}.
 \end{align}
 By the Hölder inequality, we calculate
 \begin{align}\label{HE}
     U^p(\rho)=\left(\int_\mathfrak{S}|u(\rho\sigma)|\,d\sigma\right)^p\le \left( \int_{\mathfrak{S}} 1 d\sigma \right)^{p-1} \int_\mathfrak{S}|u(\rho \sigma)|^p \, d\sigma = |\mathfrak{S}|^{p-1}\int_\mathfrak{S}|u(\rho \sigma)|^p \, d\sigma, \end{align}
 and use this in \eqref{HD} to get
 \begin{align} \label{HE5}
 &\left(\int_\mathbb{G} \left||x|^{\lambda+\mu-Q}\int_{\mathbb{B}(0,|x|)}\frac{u(y)}{|y|^\mu} dy \right|^q dx\right)^\frac{1}{q} \nonumber \\& \quad\quad\le D_{p,q,\gamma} |\mathfrak{S}|^\frac{1}{q}|\mathfrak{S}|^{\frac{p-1}{p}}\left(\int_0^\infty  \rho^{Q-1}\int_\mathfrak{s}|u(\rho \sigma)|^p\,d\sigma\,d\rho\right)^\frac{1}{p} \nonumber
 \\&\quad\quad\quad=D_{p,q,\gamma}|\mathfrak{S}|^{1+\frac{1}{q}-\frac{1}{p}}\left(\int_0^\infty  \rho^{Q-1}\int_\mathfrak{S}|u(\rho \sigma)|^p\,d\sigma\,d\rho\right)^\frac{1}{p}.
     \end{align}
Again using polar decomposition we note that 
 $$ \int_0^\infty \rho^{Q-1}\int_\mathfrak{S}|u(\rho \sigma)|^p\,d\sigma\,d\rho=\int_\mathbb{G}|u(y)|^p \,dy,$$ and consequently, \eqref{HE5} yields the following inequality  \begin{equation}
     \left(\int_\mathbb{G} \left||x|^{\lambda+\mu-Q}\int_{\mathbb{B}(0,|x|)}\frac{u(y)}{|y|^\mu} dy \right|^q dx\right)^\frac{1}{q}\le D_{p,q,\gamma}|\mathfrak{S}|^{1+\frac{1}{q}-\frac{1}{p}}\left(\int_\mathbb{G}|u(y)|^p\,dy\right)^\frac{1}{p}
 \end{equation} with $\gamma=\mu-\frac{Q-1}{p'}.$
 Finally, by choosing  $\mu=\frac{\alpha}{p},$ we note that $\frac{\beta'}{q}=\delta+\gamma-1$ and $\alpha'=\gamma p$ holds, if and only if $\lambda+\mu-Q=\frac{\beta}{q}$ holds. Indeed,
 by recalling that $\alpha'=\frac{\alpha}{Q},\, \beta'=\frac{\beta}{Q}$ and $\delta=\frac{\lambda}{Q},$ we get 
 \begin{align}
     \frac{\beta}{Qq}=\frac{\beta'}{q}= \delta+\gamma-1 \iff \frac{\beta}{q}=\delta Q+\gamma Q-Q=\lambda+\mu-Q,
 \end{align} where we have used that $\gamma Q=\frac{\alpha'}{p}Q=\frac{\alpha}{Qp}Q=\frac{\alpha}{p}=\mu \iff \gamma Q=\frac{\alpha}{p} \iff \gamma p=\frac{\alpha}{Q}=\alpha'.$
 
 Next, by replacing $\frac{u(y)}{|y|^\mu}=f(y),$
 we obtain $\gamma=\frac{\alpha}{p}-\frac{Q-1}{p'}$ and the sharp inequality 
\begin{equation}
   \left(\int_\mathbb{G} \left( \int_{\mathbb{B}(0,|x|)}f(y) dy\right)^q  |x|^\beta\, dx \right)^\frac{1}{q}\le C(p,q, Q, \alpha)\left(\int_\mathbb{G}|f(y)|^p|y|^\alpha \,dy\right)^\frac{1}{p}  
\end{equation} for $1<p<q<\infty$ with the sharp constant 
\begin{align*}
   C(p,q, Q, \alpha)&:= D_{p,q,\gamma}|\mathfrak{S}|^{1+\frac{1}{q}-\frac{1}{p}}\\& =|\mathfrak{S}|^{1+\frac{1}{q}-\frac{1}{p}}\left ( \frac{p-1}{p-1-\gamma p}\right)^{\frac{1}{p'}+\frac{1}{q}} \left( \frac{p'}{q}\right)^\frac{1}{p} \left( \frac{\frac{q-p}{p}\Gamma \left(\frac{pq}{q-p} \right)}{\Gamma \left( \frac{p}{q-p} \right)\Gamma \left( \frac{p(q-1)}{q-p}\right)}\right)^{\frac{1}{p}-\frac{1}{q}}.
\end{align*}
By using the value of $\gamma$, i.e., $\gamma=\frac{\alpha}{p}-\frac{Q-1}{p'},$ we calculate
$$p-1-\gamma p=p-1-\left(\frac{\alpha}{p}-\frac{Q-1}{p'} \right)p=Q(p-1)-\alpha.$$
Therefore, by putting the value of $p-1-\gamma p$, we get the sharp constant
\begin{equation}
    C(p,q, Q, \alpha) =|\mathfrak{S}|^{1+\frac{1}{q}-\frac{1}{p}}\left ( \frac{p-1}{Q(p-1)-\alpha }\right)^{\frac{1}{p'}+\frac{1}{q}} \left( \frac{p'}{q}\right)^\frac{1}{p} \left( \frac{\frac{q-p}{p}\Gamma \left(\frac{pq}{q-p} \right)}{\Gamma \left( \frac{p}{q-p} \right)\Gamma \left( \frac{p(q-1)}{q-p}\right)}\right)^{\frac{1}{p}-\frac{1}{q}},
\end{equation} when $1<p<q<\infty.$

For the case $p=q,$ the proof follows verbatim to the proof of the case $p<q$ above. The value of the best constant $C(p,p,Q,\alpha)$ is given by 
\begin{align}
    C(p,p,Q,\alpha)= D_{p,p,\gamma} |\mathfrak{S}|,
\end{align} where $D_{p,p,\gamma}=\lim_{p \rightarrow q} D_{p,q,\gamma} =\frac{p}{p-1-\gamma p}. $ Therefore, 
\begin{equation}
  C(p,p,Q,\alpha)=\lim_{p\rightarrow q} C(p,q,Q,\alpha)=|\mathfrak{S}| \lim_{p \rightarrow q} D_{p,q,\gamma}=|\mathfrak{S}| \frac{p}{p-1-\gamma p}
\end{equation} where $\gamma=\frac{\alpha}{p}-\frac{Q-1}{p'}.$ 
\newline By calculating the value of $p-1-\gamma p$  and using $\gamma=\frac{\alpha}{p}-\frac{Q-1}{p'},$ we get $p-1-\gamma p=Q(p-1)-\alpha$ as above. Therefore, we obtain 
\begin{equation}
    C(p,p,Q,\alpha)= \frac{p|\mathfrak{S}|}{Q(p-1)-\alpha},  
\end{equation} completing the proof of the Theorem \ref{sharphardy}. \end{proof}

\begin{rem}
It is clear from the proof of Theorem \ref{sharphardy} that the constant $C(p,q,Q,\alpha)$ in \eqref{sharphardy1} is sharp because the constant $D_{p,q, \gamma}$ in the one dimensional  inequality \eqref{S} is sharp. 
\end{rem}

The next result is the sharp weighted conjugate integral Hardy inequality on homogeneous groups.
\begin{theorem}
Let $\mathbb{G}$ be a homogeneous group of homogeneous dimension $Q$ equipped with a quasi norm $|\cdot|.$ Let $1<p\leq q<\infty$ and let $\alpha,\beta \in \mathbb{R}$. Then the following inequality  
\begin{align}
    \left( \int_{\mathbb{G}}  \left( \int_{\mathbb{G} \backslash \mathbb{B}(0,\, |x| )} u(y)\, dy\right)^q |x|^{\beta} dx\right)^{\frac{1}{q}} \leq C(p, q, Q,\alpha) \left(\int_{\mathbb{G}} |u(x)|^p |x|^{\alpha} dx \right)^{\frac{1}{p}}
\end{align} holds for all mearurable functions $u$ on $\mathbb{G}$ if and only if 
$$\alpha>Q(p-1)\quad \text{and}\quad q(\alpha+Q)-p(\beta+Q)=pqQ.$$ Moreover, the constant $C(p, q, Q,\alpha)$ is sharp and given by 
$$C(p,q,Q,\alpha)= |\mathfrak{S}|^{1+\frac{1}{q}-\frac{1}{p}} \left( \frac{p-1}{\alpha-Q(p-1)}\right)^{\frac{1}{p'}+\frac{1}{q}}\left(\frac{p'}{q}\right)^{\frac{1}{p}} \left( \frac{\frac{q-p}{p}\Gamma \left(\frac{pq}{q-p} \right)}{\Gamma \left( \frac{p}{q-p} \right)\Gamma \left( \frac{p(q-1)}{q-p}\right)}\right)^{\frac{1}{p}-\frac{1}{q}}$$ when $q>p$ and 
$$C(p, p, Q, \alpha)= \lim_{p\rightarrow q} C(p,q,Q,\alpha)=\frac{p|\mathfrak{S}|}{\alpha- Q(p-1)}$$
when  $p=q$, where $\mathfrak{S}= \{x \in \mathbb{G}: |x|=1\} \subset \mathbb{G}$ is the unit sphere with respect to the quasi-norm $|\cdot|$, and $|\mathfrak{S}|$ denotes the measure of the unit sphere $\mathfrak{S}$ in the homogeneous group  $\mathbb{G}$ with respect to the quasi-norm $|\cdot|$.
\end{theorem}
\begin{proof}
The proof is similar to the proof of Theorem \ref{sharphardy} using one dimensional sharp weighted conjugate Hardy inequality \eqref{EQ:Hardy4} from Theorem \ref{Dim1} instead of one dimensional sharp Hardy inequality \eqref{EQ:Hardy3} from Theorem \ref{Dim1} .

\end{proof}

\section*{acknowledgement} 

MR is supported by the FWO Odysseus 1 grant G.0H94.18N: Analysis and Partial
Differential Equations, the Methusalem programme of the Ghent University Special Research
Fund (BOF) (Grant number 01M01021) and by EPSRC grant
EP/R003025/2. AS is supported by UGC Non-NET fellowship.


\begin{thebibliography}{H8}








\bibitem{D}
 E.~B. Davies. 
\newblock A review of Hardy inequalities. 
\newblock {\em Operator Theory: Advances and Applications} In The Maz’ya anniversary collection, vol.
2 (Rostock, 1998).vol. 110, pp. 55–67. Basel, Switzerland: Birkhäuser.

\bibitem{EE} D.~E. Edmunds and W.~D. Evans.  \newblock Hardy operators, function spaces and embeddings. 
\newblock{\em Springer
Monographs in Mathematics. Berlin, Germany: Springer}. 2004.


\bibitem{FR} V. Fischer and M. Ruzhansky. 
\newblock Quantization on nilpotent Lie groups. 
\newblock{\em Progress in Mathematics.} vol. 314. Basel, Switzerland: Birkhäuser. (open access book).2016.

	\bibitem{FS74}
G.~B. Folland and E.~M. Stein.
\newblock Estimates for the $\overline{\partial_{b}}$ complex and
analysis on the Heisenberg group.
\newblock {\em Comm. Pure Appl. Math.}, 27:429--522, 1974.


\bibitem{FS} G.~B. Folland and E.~M. Stein.  
\newblock Hardy spaces on homogeneous groups.  
\newblock{\em Princeton University Press.} Mathematical Notes, vol. 28.
Princeton. 1982.




\bibitem{FL12}
R.~L.~Frank and E.~H.~Lieb.
\newblock Sharp constants in several inequalities on
the Heisenberg group.
\newblock {\em Ann. of Math.}, 176:349--381, 2012.
 
\bibitem{HD1} G. H. Hardy.
\newblock Notes on a theorem of Hilbert.
\newblock {\em Math. Z. 6 (1920)}, 314–317.
\bibitem{HD2}G. H. Hardy.
\newblock Notes on some points in the integral calculus, LX. A. inequality between integrals.
\newblock {\em Messenger Math. 54 (1925)} 150–156.

\bibitem{HD3} G. H. Hardy.
\newblock Notes on some points in the integral calculus, LXI. further inequalities betwen integrals.
\newblock{\em Messenger Math. 57 (1927)} 12–16.




\bibitem{KM1} V.  Kokilashvili, A.  Mekshi and L.-E. Persson.
\newblock Weighted norm inequalities for integral transforms with product kernels.
\newblock{\em Nova Science Publishers, New York} 2010.


\bibitem{KPS} A. Kufner,  L.-E. Persson and N. Samko.  2017
\newblock Weighted inequalities of Hardy type, 
\newblock{\em 
Hackensack, NJ: World Scientific Publishing Co. Pte. Ltd. 2nd edn.} 2003.




\bibitem{KP} A. Kufner and L.-E. Persson. 
\newblock Weighted inequalities of Hardy type.
\newblock{\em World Scientific Publishing Co. Inc., River Edge, NJ}. 2003.



\bibitem{KM2} A. Kufner, L. Maligranda, and L.-E. Persson. 
\newblock The prehistory of the Hardy inequality.
\newblock{\em Amer.Math. Monthly 113 (2006)}. 715–732.

\bibitem{KM3} A. Kufner, L. Maligranda, and L.-E. Persson. 
\newblock The Hardy Inequality – About its History and Some Related Results, Pilsen, 2007.






\bibitem{OK} B. Opic and A. Kufner. \newblock Hardy-type inequalities. \newblock{\em Pitman Research Notes in Mathematics Series} 1990.
vol. 219. Harlow, UK: Longman Scientific and Technical.

\bibitem{PS}
L.-E. Persson and S. G. Samko.
\newblock A Note On The Best Constants In Some Hardy Inequalities.
\newblock{\em dx.doi.org/10.7153/jmi-09-37}.



 
 
\bibitem{RS1}M. Ruzhansky and D. Suragan. 
\newblock Hardy and Rellich inequalities, identities, and sharp
remainders on homogeneous groups. 
\newblock {\em Adv. Math.} 2017 , 317, 799–822. (doi:10.1016/j.aim.2017.07.020).

\bibitem{RS} M. Ruzhansky and D. Suragan. Hardy inequalities on homogeneous groups: 100 years
of Hardy inequalities. Progress in Math., Vol. 327, Springer Birkh$\ddot{\text{a}}$user, Cham, Switzerland, 2019.
xvi+571pp.

\bibitem{RV}
M. Ruzhansky and D. Verma. Hardy inequalities on metric measure spaces.
{\em Proc. R. Soc. A.,} 475:20180310, 2019, 15pp.

\bibitem{RV1} M. Ruzhansky and D. Verma. Hardy inequalities on metric measure spaces, II: The case p>q. {\em Proc. R. Soc. A}, 477:20210136, 2021, 16pp.


\bibitem{RY}
M.~Ruzhansky and N.~Yessirkegenov.
\newblock Hypoelliptic functional inequalities.
\newblock {\em arXiv:1805.01064v1}, 2018.







\end{thebibliography}
\end{document}